\documentclass[11pt]{article}

\usepackage[pdftex]{hyperref}
\usepackage[expansion]{microtype}
\usepackage[pdftex]{color,graphicx}

\usepackage[utf8]{inputenc}
\usepackage[english]{babel}
\usepackage[T1]{fontenc}
\usepackage{lmodern}

\usepackage[a4paper,vmargin={4cm,4cm},hmargin={2.5cm,2.5cm}]{geometry}
\usepackage{mathpazo}
\usepackage[font=sf, labelfont={sf,bf}, margin=1cm]{caption}

\usepackage{amsmath,amsfonts,amssymb,amsthm,mathrsfs,amsrefs}

\newtheorem{theorem}{Theorem}
\newtheorem{lemma}[theorem]{Lemma}

\newtheorem*{remark}{Remark}

\def\E{{\mathbb E}}
\def\P{{\mathbb P}}

\def\N{{\mathbb N}}

\newcommand{\ER}{Erd\H{o}s-R\'enyi }

\hypersetup{
    pdfauthor   = {N. Enriquez, G. Faraud, L. Menard},%
    pdftitle    = {Limiting shape of the DFS on an Erdos-Renyi random graph}}

\begin{document}


\title{Limiting shape of the Depth First Search tree\\ in an \ER graph}
\date{}
\author{\textsc{Nathana\"el Enriquez}, \, \textsc{Gabriel Faraud} \, and \, \textsc{Laurent M\'enard}}

\maketitle

\centerline{\it Universit\'e Paris Nanterre}

\bigskip
\bigskip
\bigskip

{\leftskip=2truecm \rightskip=2truecm \baselineskip=15pt \small

\begin{abstract}
We show that the profile of the tree constructed by the Depth First Search Algorithm in the giant component of an Erd\H{o}s-R\'enyi graph with $N$ vertices and connection probability $c/N$ converges to an explicit deterministic shape. This makes it possible to exhibit a long non-intersecting path of length $\left( \rho_c - \frac{\mathrm{Li}_2(\rho_c)}{c} \right) \times N$, where $\rho_c$ is the density of the giant component.
\bigskip

\bigskip

\end{abstract}

} 

\bigskip
\bigskip


\section{Introduction}

The celebrated Erd\H{o}s-Renyi model of random graphs \cite{ER} exhibits a phase transition when the average degree in the graph is $1$. Above this threshold, the graph contains with high probability a unique connected component of macroscopic size called the giant component. The geometry of this giant component has been the subject of numerous research articles and we refer to the monographs by Bollob\'as \cite{B}, Durrett \cite{D} or Frieze and Karo\'nski \cite{F} for extensive surveys. Some results are striking by their sharpness. This is the case for the typical distance between vertices (see Durrett \cite{D}) and the diameter (see Riordan and  Wormald \cite{RW}) which are both of logarithmic order in the number of vertices. One could ask whether this small world effect prevents the graph from containing a long simple path. This is not the case and Ajtai, Koml\'os and Szemerédi \cite{AKS} proved that there exists a simple path of linear length in the supercritical regime, solving a conjecture by P. Erd\H{o}s \cite{E}. In the recent paper \cite{KS}, Krivelevich and Sudakov propose a simple proof of the phase transition which also exhibits a simple path of linear length in the supercritical regime. However, they only show the existence of a simple path of length $\varepsilon$ times the number of vertices in the graph for some positive $\varepsilon$. Their strategy is to analyse the classical Depth First Search algorithm (DFS) we now describe informally.

For any finite graph $G$, the DFS is an exploration process on $G$. Starting at one vertex, say $v$, it jumps to any neighbor of $v$, continues to a neighbor of this new vertex and so on, with the restriction that the process is not allowed to visit a vertex twice. The process will draw a non-intersecting path in the graph, and ultimately get stuck. The rule is then to make a step back (that is towards $v$) and start exploring again. It is clear that, at any time, the set of visited edges is a tree. Eventually, the process will completely visit the connected component of $v$ and draw a spanning tree of it.

\bigskip

In this paper, we study the length of the longest simple path constructed by the DFS when it is started at a vertex belonging to the giant component (Theorem~\ref{th:longest}). In fact, we even get the scaling limit of the spanning tree constructed by the DFS (see Theorem~\ref{th:main}).
This gives an explicit lower bound for the longest simple path in the graph:
\begin{theorem}
\label{th:longest}
Let $H_N$ be the length of the longest simple path in an Erd\H{o}s-R\'enyi random graph with $N$ vertices and parameter $c/N$. Then, in probability
$$\liminf_{N\to \infty} \frac{H_{N}}{N} \ge \rho_c-\frac{Li_2(\rho_c)}{c},$$
where $Li_2$ stands for the dilogarithm function and $\rho_c$ is the survival probability of a Galton-Watson branching process with Poisson($c$) offspring distribution characterized by the equation
\begin{equation}
\label{relrho}
1 - \rho_c = \exp (-c \, \rho_c ).
\end{equation}
\end{theorem}

\bigskip

It is interesting to consider the behavior of this lower bound when $c$ is large. As $Li_2(1)=\pi^2/6$, we have the asymptotic expansion
$$\liminf_{N\to \infty} \frac{H_{N}}{N}\geq 1-\frac{\pi^2}{6c}+o\left( \frac 1 c \right),$$
improving the former lower bound $1-2.21/c$ derived by Fernandez de la Vega in \cite{V} as mentioned in \cite{AKS}. It is natural to ask whether the bound of Theorem \ref{th:longest} is optimal or not. We did not find any evidence in either direction.

\section{The Depth First Search algorithm and its scaling limit}

In the following, we denote by $\mathcal{G}_N=(V_N, E_N)$ an Erd\H{o}s-R\'enyi random graph with $N$ vertices and parameter $c/N$. The vertex set of $\mathcal{G}_N$ is $V_N=\{1,2,\dots,N\}$ and for a pair $(i,j)\in \N^2, \, i \neq j$ the edge $\{i,j\}$ belongs to $E_N$ with probability $c/N$, independently of all the others. As already mentioned if $c>1$ then there is a constant $\rho_c$ such that the largest connected component of $\mathcal{G}_N$ grows asymptotically like $\rho_c \, N$ as $N$ goes to infinity, where, for any $c>1$, the constant $\rho_c$ is characterized by the fixed point equation \eqref{relrho}.

\subsection{The Depth First Search algorithm}

Let us formally define the DFS algorithm on $\mathcal G _N$ by induction. At each step we define the following objects: 
\begin{itemize}
\item $A_n$ is an ordered set of vertices, called {\it active vertices} at time $n$. With a slight abuse of notation, we will sometimes also denote by $A_n$ the unordered set of vertices of the ordered list $A_n$.
\item $a_n$ is the last element of $A_n$,
\item $S_n$ is a set of vertices, called \emph{sleeping vertices},
\item $R_n = \{1,\ldots, N\} \setminus (A_n \cup S_n)$ is also a set of vertices, called the \emph{retired vertices}.
\end{itemize}
Initially we set:
$$
\begin{cases}
A_0 &=(1),\\
S_0 &=\{2,3,\dots, N\}, \\
R_0 &=\emptyset.
\end{cases}
$$
The process stops when $A_n = \emptyset$. This occurs when $n = 2 |\mathcal C (1)| -1$, where $|\mathcal C (1)|$ is the number of vertices in the connected component of $1$ in $\mathcal G_N$. Knowing $A_n, S_n$ and $R_n$ we define $A_{n+1}, S_{n+1}$ and $R_{n+1}$ according to the following rules:
\begin{itemize}
\item If $a_n$ has a neighbor in $S_n$, we set
$$
\begin{cases}
a_{n+1} &=\inf  \{ k \in S_n: \{a_n, k\} \in E_N\},\\
A_{n+1} &=A_n\cup a_{n+1} \mbox{ (that is, the concatenation of } A_n \mbox{ and } a_{n+1}),\\
S_{n+1} &=S_n \backslash \{a_{n+1}\},\\
R_{n+1} &=R_n.
\end{cases}
$$

\item If however, $a_n$ has no neighbor in $S_n$, we set
$$
\begin{cases}
A_{n+1}&=A_n\backslash a_{n} \mbox{ (that is } A_n \mbox{ with its last element removed) },
\\ S_{n+1}&=S_n,
\\ R_{n+1}&=R_n\cup \{a_n\}.
\end{cases}
$$
\end{itemize}
The sequence of vertices $(a_n)$ is a nearest neighbor walk on the connected component of $1$ and its trace is a spanning tree of this component. Moreover, the chronology of the construction makes this tree rooted and planar. By construction, the list $A_n$ is the ancestral line between $a_n$ and $1$ in this spanning tree. The set $S_n$ is the set of vertices that have not been visited by the walk $(a_n)$ before time $n$. The vertices in $R_n$ are those for which the construction of the process ensures that they have no neighbor in $S_n$. 

\begin{remark}
From a probabilistic point of view it might seem unnatural to take the neighbor with smallest index in the definition of $(a_n)$ instead of, for example, picking a neighbor at random. As it will become clear in the proofs, this does not change the asymptotics of the process.
\end{remark}

\subsection{Scaling limit of the DFS}

At each step, the current height of the walker in the spanning tree constructed by this algorithm is denoted by $X_n = |A_n| -1$. This defines a Dyck path $X = (X_n)_{0 \leq n \leq 2 |\mathcal C (1)| -1}$: it starts at $0$, has increments in $\{ -1 , +1\}$ and is non-negative except at its final value $-1$. The process $X$ is the canonical contour process in clockwise order of the spanning tree constructed by the DFS algorithm. Because of all the information it encodes, the process $X$ will be our main object of interest. Since we are mainly interested in the geometry of the giant component of $\mathcal G_N$, we study the process $X$ conditional on the event $\mathbf S$ that $1$ belongs to the largest component of $\mathcal{G}_N$. This event has asymptotic probability $\rho_c$. Our main result is the convergence of the process $X$ to a deterministic curve, illustrated in Figure \ref{fig:limit}:

\begin{theorem}\label{th:main}
Conditional on $\mathbf S$, the following limit holds in probability for the topology of uniform convergence: 
$$\lim_{N\to \infty} \frac{X_{\lceil tN\rceil }}{N}=h(t),$$
where the function $h$ is continuous and defined on the interval $[0, 2 \rho_c]$. The graph $(t,h(t))_{t \in [0, 2 \rho_c]}$ can be divided into a first increasing part and a second decreasing part. These parts are respectively parametrized by:
\begin{align*}
(t,h(t))_{0 \leq t \leq f(0)} &= \left( f(\rho),g(\rho) \right)_{0\leq \rho \leq \rho_c},\\
(t,h(t))_{f(0) \leq t \leq 2 \rho_c} &= \left( f(\rho) + 2 \rho \left( 1 - \frac{f(\rho) + g(\rho)}{2} \right),g(\rho) \right)_{0\leq \rho \leq \rho_c},
\end{align*}
where the functions $f$ and $g$ are given by
\begin{align*}
f(\rho) &= \frac{1}{c}\left[Li_2(\rho_c)-Li_2(\rho)+\log\frac{1-\rho_c}{1-\rho}-2\left(\frac{\log(1-\rho_c)}{\rho_c}-\frac{\log(1-\rho)}{\rho} \right)\right],\\
g(\rho) &= \frac{1}{c}\left[Li_2(\rho)-Li_2(\rho_c)+\log\frac{1-\rho}{1-\rho_c}\right],
\end{align*}
and $Li_2$ stands for the dilogarithm function.
\end{theorem}
\begin{figure}[ht!]
\begin{tabular}{ccc}
\includegraphics[scale=0.24]{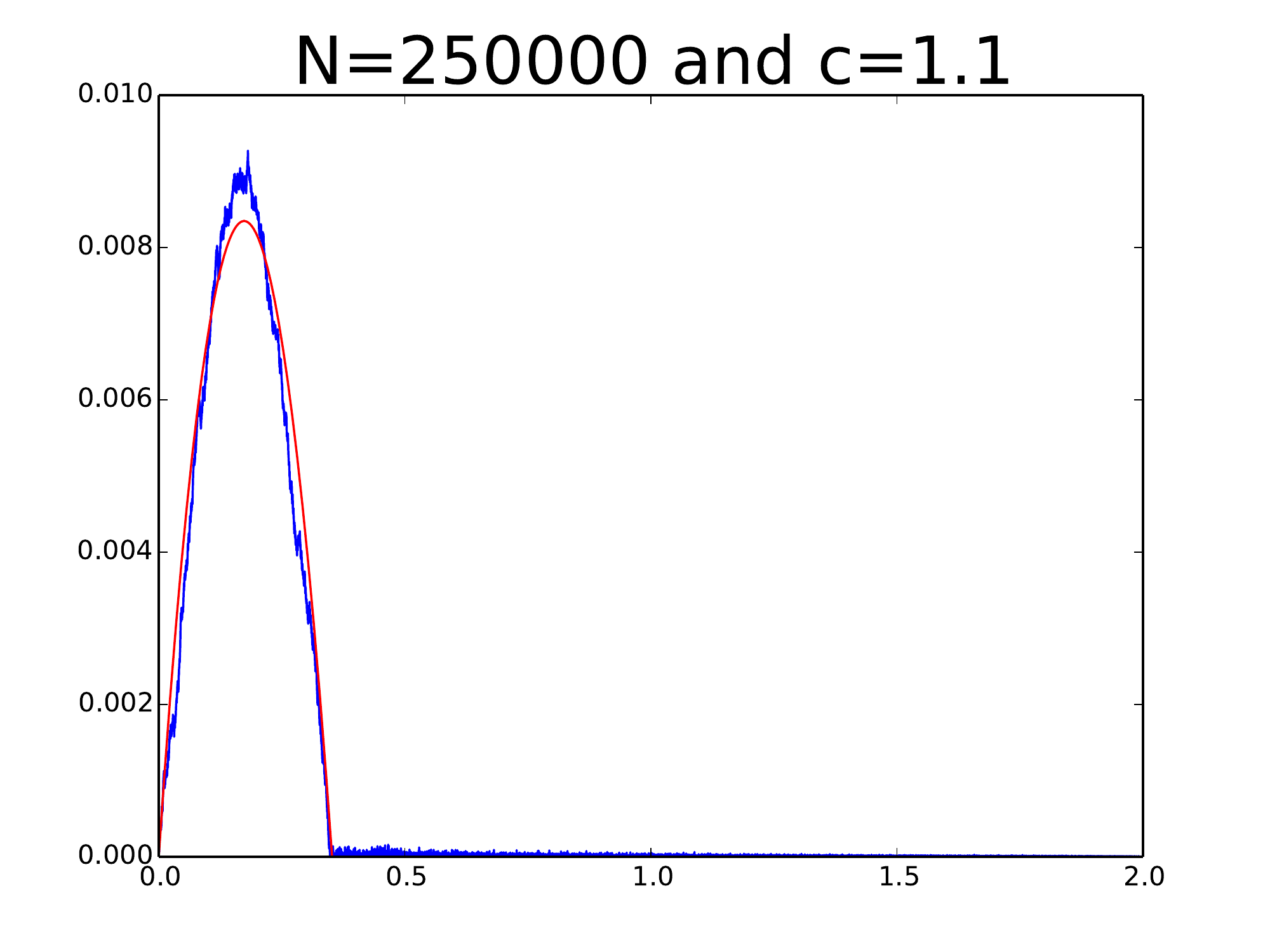}
&
\includegraphics[scale=0.24]{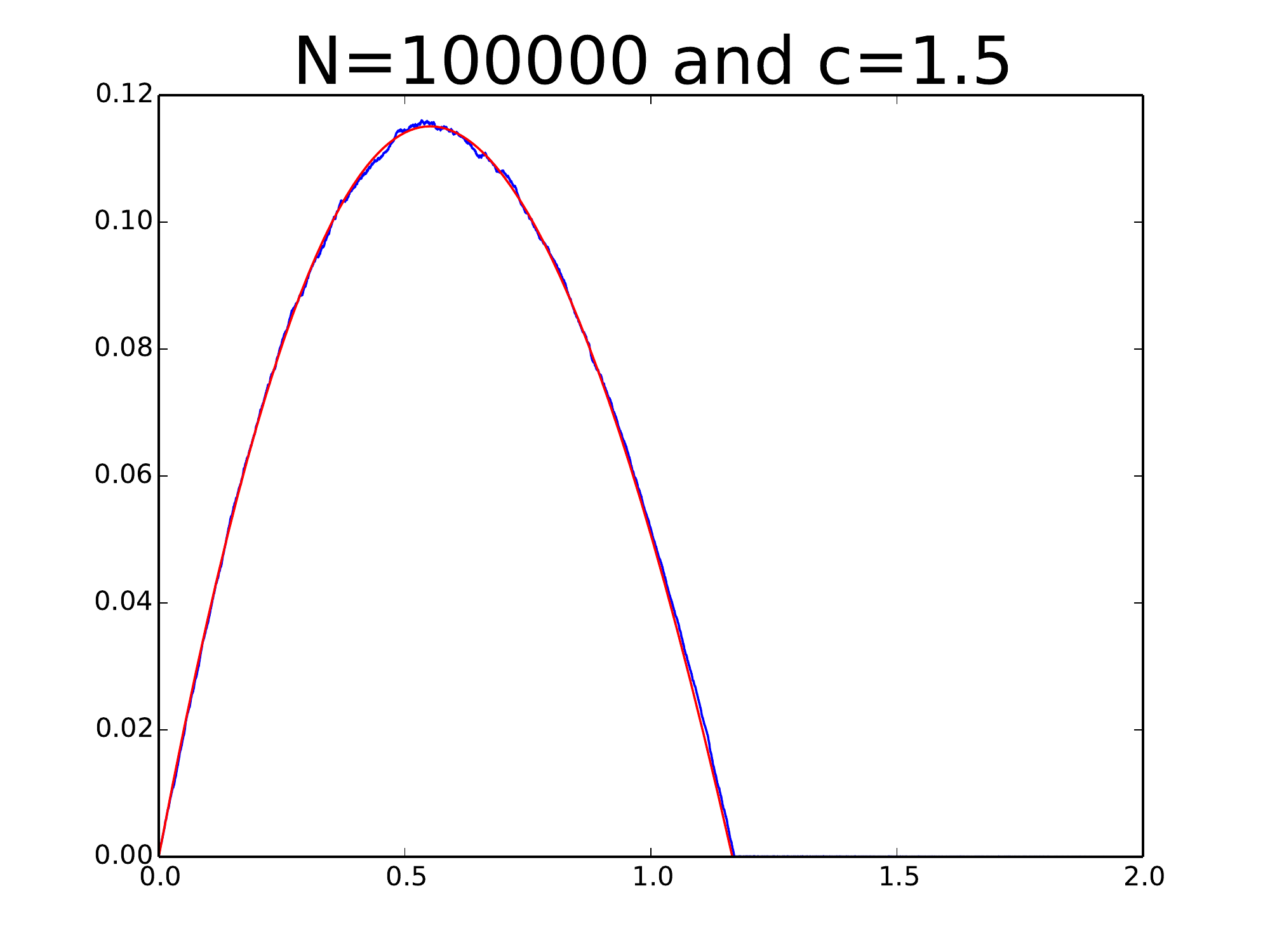}
&
\includegraphics[scale=0.24]{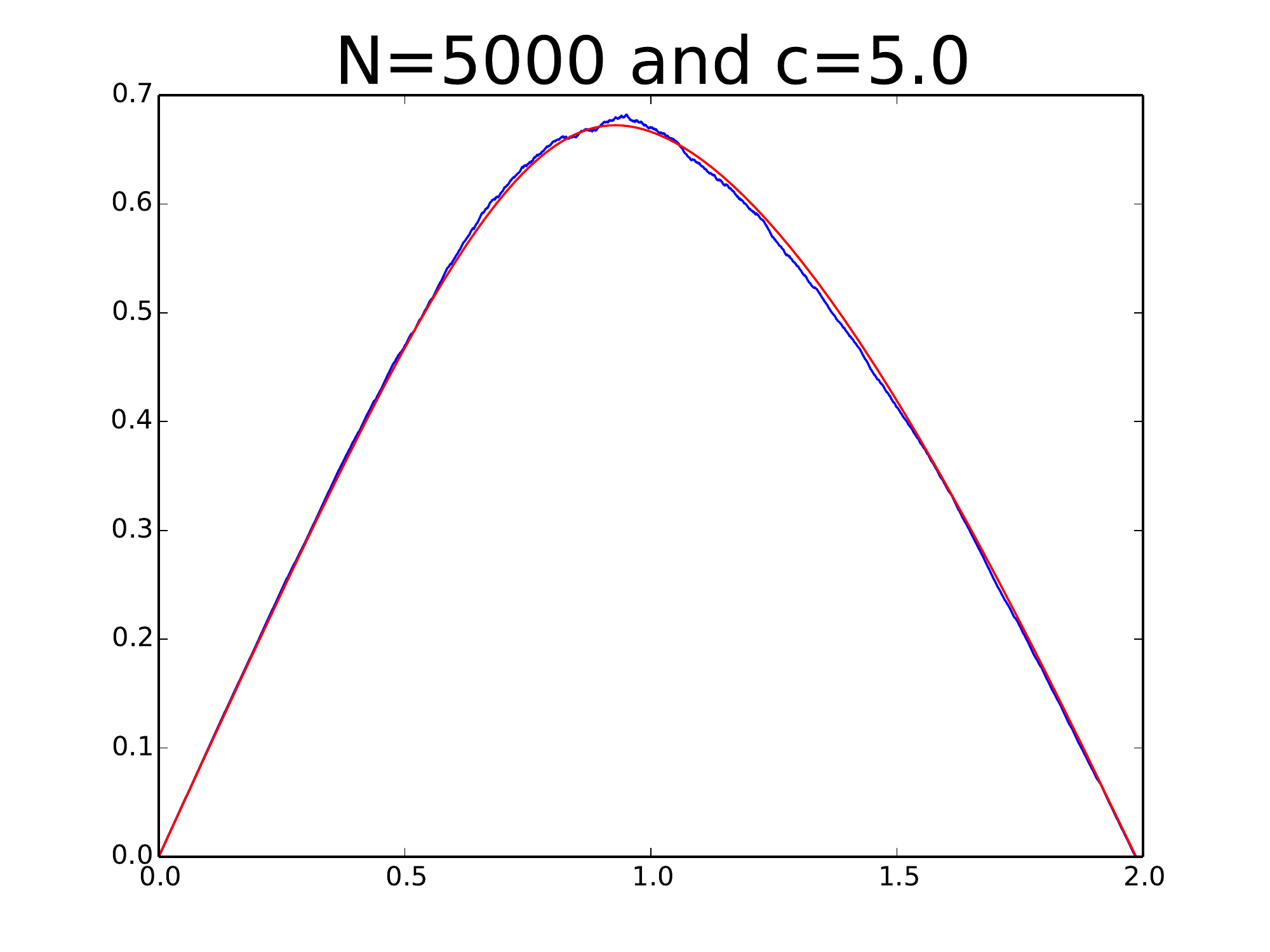}
\end{tabular}
\caption{\label{fig:limit}Simulations of $(X_{\lceil tN\rceil }/N)_{t\in [0,2]}$ (blue) and the limiting shape (red) for various values of $N$ and $c$. Notice that when $c$ is close to $1$, we have to take $N$ very large for the walk to be close to its limit.}
\end{figure}

Theorem~\ref{th:longest} is easily obtained by computing the maximal height of the curve given in Theorem~\ref{th:main}, which is equal to
\[
g(0) = \frac{1}{c}\left(\log\frac{1}{1-\rho_c} -Li_2(\rho_c) \right) = \rho_c - \frac{Li_2(\rho_c)}{c}.
\]

\section{Pseudo renewal times and strategy of the proof}

We call $\mathcal{F}_n$ the canonical filtration associated to $(a_n)$. Notice that $\mathcal{F}_n$ carries some partial information on the underlying Erd\H{o}s-R\'enyi graph but not all of it. In particular the graph structure of $S_n$ given $\mathcal{F}_n$ is that of an Erd\H{o}s-R\'enyi graph with connection probability $c/N$ since the connection between vertices of $S_n$ have not yet been tested at time $n$.

We call $\alpha_n=|A_n\cup R_n|/N$ the non-decreasing proportion of vertices explored by the process at time $n$. It is straightforward to check that $\alpha_n=\frac{X_n+n}{2N}$. Note that at time $n$, conditional on $\alpha_n$, the expected number of unexplored vertices neighboring $a_n$ is $(1-\alpha_n)c.$ Therefore it is natural to expect two successive phases:
\begin{itemize}
\item When $(1-\alpha_n)c>1$, the walker finds a lot of unexplored vertices allowing it to drift away from its starting point. We call that phase \emph{the way up}.
\item When $(1-\alpha_n)c<1$, the walker spends most of the time backtracking towards its starting point. We call that phase \emph{the way down}.
\end{itemize}

\subsection{Pseudo renewal times and the way up}\label{sec:prt}

On the way up, every time the walker visits a new vertex, there is a positive probability that this vertex belongs to the largest component of the new $S_n$. However this is not guaranteed, as the walker could be in a dead end. If this is indeed the case, the walker will soon go back to the previously visited vertex. On the other hand, if the walker is not in a dead end, it is going to spend a very long time (that is of order $N$) before returning to the current vertex, as it needs to fully explore the largest component of $S_n$. Therefore the walk $(X_n)$ contains a "spine" of macroscopic size, with small excursions.
In order to detect this spine, we introduce the following sequence of random pseudo renewal times. Let 
\begin{equation}\label{deftau}
\begin{cases}
\tau_0 &= 0, \\
\tau_{i+1} &= \inf\{n> \tau_i \, : \, X_n=i+1, \, \inf\{k \, : \, X_{n+k} = i \}>\sqrt{N}\} \wedge 2N. 
\end{cases}
\end{equation}

In words, $(\tau_i)$ is the sequence of times where the walk hits a vertex and does not come back before having visited a macroscopic portion of the graph (see Figure \ref{fig:deftaui} for an illustration). We take the minimum with $2N$ to ensure that these times are well defined even if the set $\{n> \tau_i; X_n=i+1, \, \inf\{k; X_{n+k} = i +1 \}>\sqrt{N}\}$ is empty, in which case $\tau_j = 2N$ for every $j \geq i$. However this only happens when $(1 - \alpha_n)c$ is close to $1$.

An important observation is that the $\tau_i$'s are not stopping times with respect to $\mathcal{F}_n$. However, in the large $N$ limit, they have a nice description as we will see in the following. 

\begin{figure}[ht!]
\begin{center}
\includegraphics[scale=0.9]{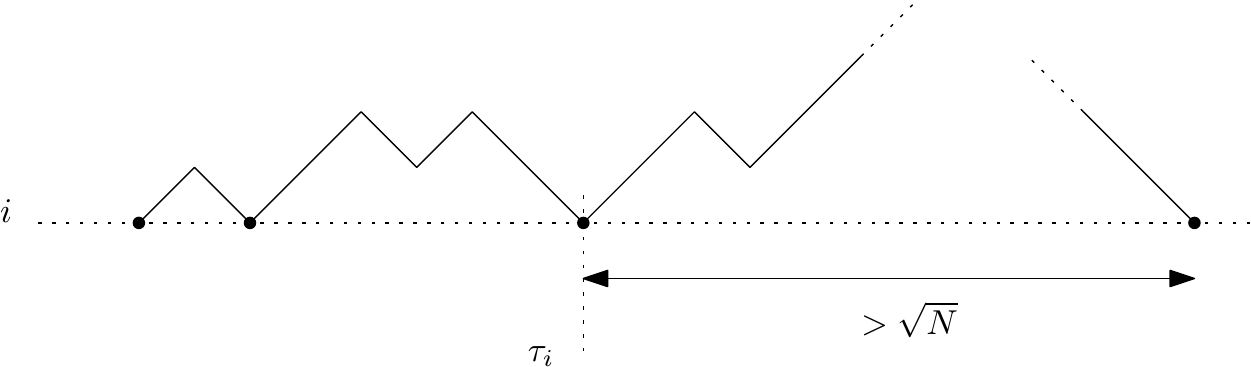}
\caption{\label{fig:deftaui}Illustration of a pseudo renewal time $\tau_i$.}
\end{center}
\end{figure}

\subsection{Strategy of the proof}
\label{sec:strat}

When hitting a pseudo renewal time $n =\tau_i$, we know that the walker is necessarily at a vertex $a_n$ belonging to the largest component of $S_{n-1}$. The neighbors of $a_n$ in $S_{n}$ are vertices picked at random, independently of the edges between vertices of $S_{n}$. Among these neighbors, some are in small components -- typically of finite size -- while at least one of them is in the largest one. Therefore, the increment $\tau_{i+1}-\tau_i$ corresponds to the time it takes to find the largest component of $S_n$. The number $C(a_n)$ of neighbors of $a_n$ in $S_n$ is close to a Poisson distribution, while the number $G(a_n)$ of tries it takes to find the largest component of $S_n$ is close to a geometric distribution, as it is a sequence of almost independent tries due to the very small amount of vertices visited between two tries. As we know that the procedure succeeds, the number of neighbors tested before finding the good one is a geometric (minus one) random variable conditioned to be smaller than a Poisson random variable. Figure \ref{fig:deltataui} gives an illustration of this situation.
\begin{figure}[ht!]
\begin{center}
\includegraphics[scale=0.9]{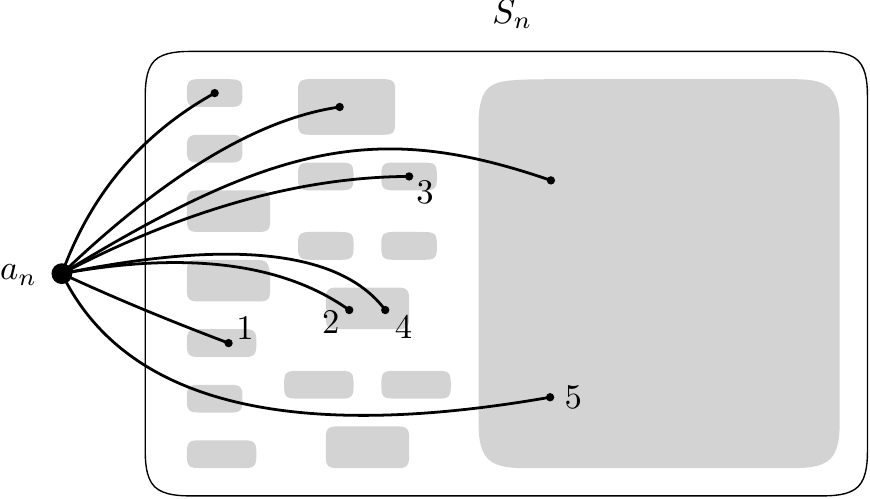}
\caption{\label{fig:deltataui}Local situation at a pseudo renewal time. Grey areas represent the connected components of $S_n$. In this example $C(a_n) = 8$ and $G(a_n) = 5$.}
\end{center}
\end{figure}

When the walker goes to a neighbor of $a_n$, it has to visit its whole connected component inside $S_n$ before returning to $a_n$. The time it takes to do so is twice the number of vertices in this connected component and will be small. Indeed, by definition, this connected component is not the giant component of the graph $S_n$ and therefore is asymptotically a subcritical Galton-Watson tree with an explicit offspring distribution.
These observations make it possible to study in detail the conditional expectation
\[
\mathbb E \left[ \tau_{i+1}-\tau_i \middle| \tau_i \right]
\]
in Section \ref{proofespcond}. A precise statement is given in Lemma \ref{espcond}.

A crucial parameter in our estimates of the above expectation is the proportion of sleeping vertices available at time $\tau_i$, that is $(1- \alpha_n)$ with our notation. In order to control this parameter, we introduce in Section~\ref{proofespcond} a sequence of random times $(h_k)$ corresponding to times where this proportion of available vertices hits fixed levels, independent of $N$.
As we already mentioned, the times $\tau_i$ are not stopping times. However, we will see in Section \ref{cht} that the $\tau_i$'s can be viewed as a Markov chain for which the $h_k$'s are stopping times. This allows us to prove a concentration result for the $h_k$'s with a martingale argument. See Lemma \ref{lemdis} for a precise statement.

The knowledge of the $h_k$'s and of the associated times $\tau_{h_k}$ provides pinning points through which the profile of the walk has to pass. Slope arguments then show that the normalized profile of the walk converges and the expectations $\mathbb E \left[ \tau_{i+1}-\tau_i \middle| \tau_i \right]$ give us access to the derivative of the increasing part of the limiting profile. The decreasing part is then deduced from the increasing one by a simple argument once one realizes that the time it takes to go back to a given level is twice the size of the giant component of the graph composed by the current sleeping vertices.
This proof of Theorem \ref{th:main} is detailed in Section \ref{pinning}.

\section{The proof itself}{\label{proof-tau}}
\subsection{Giant component among sleeping vertices}\label{sec:G}

We already mentioned in Section \ref{sec:prt} that the pseudo renewal times $\tau_i$ may degenerate. This will not be the case if, for every $n$ during the way up, the graph $S_n$ has no connected component of mesoscopic size. The next lemma shows that the probability of this event converges to $1$. For later convenience, we also include a logarithmic bound for the maximal degree in the graph.

To avoid problems at criticality, we fix a margin $\eta>0$ and consider times where $(1-\alpha_n)c>1+\eta,$ or equivalently
\begin{equation}\label{eq:alphaeta}
\alpha_n<1-(1+\eta)/c.
\end{equation}
\begin{lemma}
Let $\bf G$ be the event that, for every $n$ such that $\alpha_n$ verifies \eqref{eq:alphaeta}, the graph $S_{n}$ has no connected component of size between $N^{1/10}$ and $N^{9/10}$, and that the maximum degree of a vertex in $S_0$, hence in every $S_n$, is at most $\log N$. Then
\[
\lim_{N\to\infty}\P({\bf G})=1.
\]
\end{lemma}
\begin{proof}
The maximum degree of \ER graphs is well known (see \emph{e.g.} \cites{B,F}) and we just focus on the size of the connected components.

Recall that, by construction, for every $n$, the subgraph spanned by $S_n$ is an Erd\H{o}s-R\'enyi random graph with $(1-\alpha_{n}) N$ vertices and parameter $c/N$. 

Fix $k \geq 0$ and let $Z_k$ denote the number of connected components of size $k$ in an \ER graph of size $n$ and parameter $p$. Using the fact that a complete graph with $k$ vertices has $k^{k-2}$ spanning trees, we get:
\[
\mathbb E \left[ Z_k \right] \leq \binom{n}{k} k^{k-2} p^{k-1} (1-p)^{k(n-k)}.
\]
When $p=c/N$ and $n = (1-\alpha) N$ with $c(1+\alpha)< 1 + \eta$, using classical inequalities we obtain
\begin{align*}
\mathbb E \left[ Z_k \right] & \leq \frac{A}{\sqrt{k}} \left( \frac{(1-\alpha)N \, e}{k} \right)^k k^{k-2} \left( \frac{c}{N} \right)^{k-1} e^{-c (1-\alpha)k + c \frac{k^2}{N}} \\
& \leq \frac{A \, N}{k^{5/2}} \left( c e^{-\eta+c \frac{k}{N} } \right)^k.
\end{align*}
Now, if $k \in [N^{1/10},N^{9/10}]$ we obtain
\begin{align*}
\mathbb E \left[ Z_k \right] & \leq \frac{A \, N}{N^{5/20}} \left( c e^{-\eta+c N^{-1/10}  } \right)^k \leq A N^{3/4} \left(c e^{-\eta+c N^{-1/10}} \right)^k.
\end{align*}
If $N$ is large enough, the parameter $c$ being fixed, we have $c e^{-\eta+c N^{-1/10}} < 1$ and therefore
\begin{align*}
\mathbb E \left[ \sum_{k=N^{1/10}}^{N^{9/10}} Z_k \right]
& \leq A N^{7/4} \left( c e^{-\eta+c N^{-1/10}} \right)^{N^{1/10}}.
\end{align*}
The lemma follows from the union bound and Markov's Inequality.
\end{proof}

\subsection{The renewal increments} \label{proofespcond}

To get Theorem~\ref{th:main}, we need a good estimate of the expected difference between to consecutive pseudo renewal times. As we will see, the law of $\tau_{i+1} - \tau_i$ mainly depends on $\alpha_{\tau_i}$, therefore we introduce the random indices $(h_k)$, depending on $N$ and a fixed $\varepsilon > 0$, defined as
\begin{equation}\label{defhk}
h_k=\inf\{i:\alpha_{\tau_i}>k\varepsilon\}.
\end{equation}
These indices correspond to heights for the walk $(X_n)$ by the relation $X_{\tau_{h_k}} = h_k$. The points $(\tau_{h_k}, h_k)$ will be our \emph{pinning points} for the profile of the walk.

The $h_k$'s are well-defined during the way up, at least for times $n$ such that $(1-\alpha_n)c>1+\eta$.
This corresponds to
\begin{equation}
\label{eq:K}
k\le K := (1-(1+\eta)/c)/\varepsilon.
\end{equation}
The fact that the parameter $\alpha$ varies only slightly between two consecutive $h_k$'s means that the sequence $(\tau_{i+1}-\tau_i)_{h_k\le i \le h_{k+1}}$ is almost an i.i.d. sequence.

\begin{lemma}\label{espcond}
There exists a constant $C$ such that, if $N$ is large enough, for every integer $i \in [h_k , h_{k+1}[$ with $k\le K$, one has
$$\frac{2}{\rho_{(1- k \varepsilon)c}}-1-C\varepsilon\le \E[\tau_{i+1}-\tau_i|\tau_i]\le \frac{2}{\rho_{(1-k \varepsilon)c}}-1+C\varepsilon.$$
\end{lemma}
\begin{proof}
To be able to bound the conditional expectation of $\tau_{i+1} - \tau_i$, we need to introduce the fundamental decomposition of the trajectory of $(X_n)$ during this interval, leading to identity~\eqref{IMPORTANT} below.
At time $n$, the walker is at a vertex $a_n$ having $C(a_n)$ neighbors inside $S_n$ (see Figure \ref{fig:deltataui}). The law of $C(a_n)$ is complicated unless the time $n$ is the first visit of $a_n$. Indeed, for for such a time $n$, the algorithm has never tested the connection between vertices of $S_n$ and $a_n$, meaning that the integer $C(a_n)$ is just a binomial random variable with parameters $(1-\alpha_n)N$ and $c/N$. We denote by ${\bf F}_n$ the event that $n$ is the first visit to $a_n$. In addition, notice that, on the event ${\bf F}_n$, the number $C(a_n)$ and the neighbors $x_1 < \cdots < x_{C(a_{n})}$ of $a_n$ in $S_n$ are independent of the connections inside $S_n$.

For every $n$, call ${\bf H}_n$ the event that the return time to $a_{n-1}$ is at least $\sqrt{N}$. On $\mathbf F_n$, this is equivalent to the fact that the connected component of $a_n$ in $S_{n-1}$ has at least $\sqrt{N}/2$ vertices meaning that $\{n=\tau_i\}=\{X_n=i\}\cap{\bf F}_n\cap{\bf H}_n$.

We denote by $G(a_n)$ the smallest $k$ such that the connected component of $x_{k+1}$ in $S_n$ has size larger than $\sqrt N/2$, and $G(a_n) = C(a_n)$ if none of the $x_i$'s is in such a connected component. 
For $1\le i <G(a_n)$, we call $W_i$ the number of vertices in the connected component of $x_i$ in $S_n$. We fix however $W_i=0$ if $x_i$ belongs to the connected component of a previously explored neighbor, meaning $x_i$ will be retired before the algorithm has the chance to test the connection between $a_n$ and $x_i$ (see for example the vertices number $2$ and $4$ in Figure~\ref{fig:deltataui}).

On the event $\{X_n=i\}\cap{\bf F}_n \cap \mathbf G$, the event $\mathbf H_n$ is equivalent to the fact that the connected component of $a_n$ in $S_{n-1}$ contains at least $N^{9/10}$ vertices. Using the bound on the maximal degree in the graph given by $\mathbf G$, this is also equivalent to the fact that at least one of the neighbors of $a_n$ in $S_{n-1}$ has a connected component in $S_n$ of size at least $N^{9/10}/\log{N}$, or $\sqrt{N}/2$. Therefore, on the event $\{X_n=i\}\cap{\bf F}_n \cap \mathbf G$, the event $\mathbf H_n$ is equivalent to $G(a_n)<C(a_n)$ and 
\begin{equation}\label{IMPORTANT}
\tau_{i+1}-\tau_i=1+2\sum_{j=1}^{G(a_n)} W_j.
\end{equation}
Conditional on $\tau_i$, the distribution of $(G(a_n),(W_j)_{ 1\le j\le G(a_n)})$ is explicit and only depends on $\alpha_{\tau_i}=\frac{i + \tau_i}{2N}$. Therefore we have shown that, on the event ${\bf G}$, the sequence $(\tau_i)$ is coupled with a non-homogeneous Markov chain, and the $h_k$'s are stopping times for this Markov chain. 

\bigskip

We can now turn to the actual proof of the lemma. We assume that $\varepsilon$ is small enough and that $N$ is large enough. In all our computations, $C$ denotes a constant independent on $k$, $N$ and $\varepsilon$ which can change from line to line to keep computations easier to read.

Recall $h_k\le i < h_{k+1}$, meaning that
\[k\varepsilon:=\alpha_- \leq \alpha_{\tau_i} <  \alpha_{\tau_{i+1}}<\alpha_+:=(k+1)\varepsilon + \varepsilon^2.
\]
Indeed, on the event ${\bf G}$, the difference $\tau_{i+1}-\tau_i$ is at most $N^{1/10}\log N$. Therefore, if $N$ is large enough, we can make sure that the variation in $\alpha$ between two subsequent $\tau_i$'s stays arbitrarily small.

Dropping the dependency in $N$, we call $p_\alpha$ the probability that a randomly taken vertex in an \ER graph with $(1-\alpha)N$ vertices and parameter $c/N$  belongs to a connected component of size larger than $\sqrt N/2$. By Dini's theorem, the sequence $p_{\alpha}$ converges uniformly to $\rho((1-\alpha)c)$ as $N$ goes to infinity. We want to compute 
\begin{equation}\label{condexp}
\E\left[\sum_{j=1}^{G(a_n)} W_j \middle| G(a_n)<C(a_n)\right]=\frac{{\displaystyle \sum_{k=0}^\infty \E\left[\sum_{j=1}^{k} W_j \, {\bf 1}_{G(a_n)=k}{\bf 1}_{C(a_n)>k}\right]}}{{\displaystyle \P \left( G(a_n)<C(a_n) \right)}}.
\end{equation}
For a fixed $k$,
$$\E\left[\sum_{j=1}^{k} W_j \, {\bf 1}_{G(a_n)=k} \, {\bf 1}_{C(a_n)>k}\right]=\E\left[\sum_{j=1}^{k} W_j \, \P\left(G(a_n)=k \text{ and } C(a_n)>k \, \middle| \, \sum_{j=1}^{k} W_j\right)\right].$$

Conditional on $(W_l)_{l\le k}$, if $\bigcap_{l\le k}\{W_l<\sqrt{N}\},$ the event $\{G(a_n)=k\}$ means that $x_{k+1}$ belongs to a large component of $S_n$. This is also true after removing the components of $x_1,\dots, x_k$ to get rid of dependencies. Besides, $\{C(a_n)>k\}$ means that $a_n$ has at least $k+1$ children. By independence between the neighbors of $a_n$ and the connections inside $S_n$
$$
\E\left[\sum_{j=1}^{k} W_j \, \P\left(G(a_n)=k \text{ and } C(a_n)>k \, \middle| \, \sum_{j=1}^{k} W_j \right) \right]\le\E\left[\sum_{j=1}^{k} W_j \, {\bf1}_{\bigcap_{l\le k}\{W_l<\sqrt{N}\}}\right]\P(C(a_n)>k) \,p_{\alpha_-}
$$
and
$$\E\left[\sum_{j=1}^{k} W_j \, \P\left(G(a_n)=k \text{ and } C(a_n)>k \, \middle| \, \sum_{j=1}^{k} W_j\right)\right]\ge\E\left[\sum_{j=1}^{k} W_j {\bf1}_{\bigcap_{l\le k}\{W_l<\sqrt{N}\}}\right]\P(C(a_n)>k)\, p_{\alpha_+}.$$

We turn to the expectation in the last bounds.
\begin{multline}
\label{eq:geom}
\E \left[\sum_{j=1}^{k} W_j \, {\bf1}_{\bigcap_{l\le k}\{W_l<\sqrt{N}\}}\right]=\sum_{j=1}^{k}\P\left(\bigcap_{l\le j-1}\{W_l<\sqrt{N}\}\right) \times\\
\E\left[W_j \, {\bf 1}_{W_j<\sqrt{N}} \, \middle| \, \bigcap_{l\le j-1}\{W_l<\sqrt{N}\}\right]\P\left(\bigcap_{j+1\le l\le k}\{W_l<\sqrt{N}\}\, \middle| \, \bigcap_{l\le j}\{W_l<\sqrt{N}\}\right).
\end{multline}
Using once again the fact that the local value of $\alpha$ remains between $\alpha_-$ and $\alpha_+$ with high probability,
\[
\E\left[\sum_{j=1}^{k} W_j \, {\bf1}_{\bigcap_{l\le k}\{W_l<\sqrt{N}\}}\right] \geq
(1-p_{\alpha_-})^k \, \E\left[W_j \, \middle| \, \bigcap_{l\le j}\{W_l<\sqrt{N}\}\right]\phantom{.}
\]
and
\[
\E\left[\sum_{j=1}^{k} W_j \, {\bf1}_{\bigcap_{l\le k}\{W_l<\sqrt{N}\}}\right] \leq
(1-p_{\alpha_+})^k \, \E\left[W_j\, \middle|\, \bigcap_{l\le j}\{W_l<\sqrt{N}\}\right].
\]

Conditional on $\bigcap_{l\le j}\{W_l<\sqrt{N}\}$, the random variable $W_j$ is either the size of a small component in an Erd\H{o}s-R\'enyi random graph with parameter $c$ and a number of vertices between $(1-\alpha_+)N$ and $(1-\alpha_-)N$, or zero if $x_j$ belongs to one of the previously visited components, which has probability smaller than $\varepsilon$ for $N$ large enough. The expected size of a small component in an Erd\H{o}s-R\'enyi random graph with parameter $c$ and $(1-\alpha)N$ vertices converges, for a fixed $\alpha$, to the expected size of a Galton-Watson tree with Poisson$((1-\alpha)c)$ offspring distribution conditioned on extinction. This, in turn, is a subcritical Galton-Watson tree with Poisson$((1-\rho_{(1-\alpha)c})(1-\alpha)c)$ offspring distribution having expected size
$$\frac{1}{1-(1-\rho_{(1-\alpha)c})(1-\alpha)c}.$$
Using the smoothness of $\rho_x$ as a function of $x$, for $N$ large enough and for every $\alpha\in[\alpha_-,\alpha_+],$ the expected size of a small component is thus in the interval 
\[
\left[
\frac{1}{1-(1-\rho_{(1-\alpha_-)c})(1-\alpha_-)c}-C\varepsilon \, ; \, 
\frac{1}{1-(1-\rho_{(1-\alpha_-)c})(1-\alpha_-)c}+C\varepsilon
\right].
\]

Equation~\eqref{condexp} then gives
\begin{multline}\label{upperbound}\E\left[\sum_{j=1}^{G(a_n)} W_j\, \middle| \, G(a_n)<C(a_n)\right]\\ \le\sum_{k=0}^{\infty}k \, p_{\alpha_-}(1-p_{\alpha_+})^k\left(\frac{1}{1-(1-\rho_{(1-\alpha_-)c})(1-\alpha_-)c}+C\varepsilon\right)\frac{\P(C(a_n)>k)}{\P(G(a_n)<C(a_n))}\end{multline}
and
\begin{multline*}\E\left[\sum_{j=1}^{G(a_n)} W_j \, \middle| \, G(a_n)<C(a_n)\right]\\ \ge\sum_{k=0}^{\infty}k \, p_{\alpha_+}(1-p_{\alpha_-})^k\left(\frac{1}{1-(1-\rho_{(1-\alpha_-)c})(1-\alpha_-)c}-C\varepsilon\right)\frac{\P(C(a_n)>k)}{\P(G(a_n)<C(a_n))}.\end{multline*}

As both bounds will be treated similarly, we will focus on the upper bound \eqref{upperbound}. By a coupling argument we can always assume that for $N$ large enough, with high probability, the random variable $C(a_n)$ is larger than a Poisson$(c(1-\alpha_-))$ random variable, denoted by X in the following. 
We call
$$S_n:=\sum_{k=0}^{n}k(1-p_{\alpha_+})^{k-1}=\frac{1-(1-p_{\alpha_+})^{n+1}}{p_{\alpha_+}^2}-\frac{(n+1)(1-p_{\alpha_+})^{n}}{p_{\alpha_+}}.$$
Isolating the sum in \eqref{upperbound}, we compute
\begin{align}
\sum_{k=0}^\infty k \, (1-p_{\alpha_+})^k \, \P(C(a_n)>k)&\le (1-p_{\alpha_+}) \, \sum_{k=1}^\infty (S_{k}-S_{k-1}) \, P(X\ge  k+1)+C\varepsilon \notag\\
&=(1-p_{\alpha_+})\sum_{k=0}^\infty S_{k} \, P(X=k+1)+C\varepsilon \notag\\
&=(1-p_{\alpha_+}) \,\E(S_{X-1})+C\varepsilon, \label{eq:espSp}
\end{align}
with the convention $S_{-1}=0$. It is straightforward to check
\begin{equation}
\label{eq:SX}
\E(S_{X-1})=\frac{1}{p_{\alpha_+}^2}\left[1-\exp(-(1-\alpha_-)cp_{\alpha_+})-p_{\alpha_+}(1-\alpha_-)c\exp((-(1-\alpha_-)cp_{\alpha_+})\right].
\end{equation}
For any $\alpha_-$, the function on the right hand side of \eqref{eq:SX} is infinitely differentiable and therefore Lipschiz in $p_{\alpha_+}$. In addition, the Lipschitz coefficient of this function can be computed explicitely and bounded uniformly in $\alpha_-$.

By uniform convergence  $|p_{\alpha_+}-\rho_{(1-\alpha_-)c}|\le \varepsilon$ and $|p_{\alpha_-}-\rho_{(1-\alpha_-)c}|\le \varepsilon$ if $N$ is large enough. Therefore we can replace $p_{\alpha_+}$ by $\rho_{(1-\alpha_-)c}$ in the previous computation, with only a error of order $C\varepsilon.$
Recalling relation \eqref{relrho} characterizing $\rho$, we obtain 
\begin{align}
\E(S_{X-1})&\le \frac{1}{\rho_{(1-\alpha_-)c}^2}\left[1-(1-\rho_{(1-\alpha_-)c})-\rho_{(1-\alpha_-)c}(1-\alpha_-)c(1-\rho_{(1-\alpha_-)c})\ \right]+C\varepsilon \notag\\
&\le \frac{1}{\rho_{(1-\alpha_-)c}}\left[1-(1-\alpha_-)c(1-\rho_{(1-\alpha_-)c})\right] + C \varepsilon. \label{eq:espSX}
\end{align}

We turn to the factor $\P(G(a_n)<C(a_n))$. Denote by $Y$ a Poisson$(c(1-\alpha_+))$ random variable. Using once again a coupling argument as well as the same decomposition as when dealing with the first member of \eqref{eq:geom} we get
\begin{align*}
\P(G(a_n)<C(a_n))&\ge\P(G(a_n)<Y)-C\varepsilon\ge 1- \E[(1-p_{\alpha_-})^Y]-C\varepsilon\\
&\ge 1-\exp{(c(1-\alpha_+)p_{\alpha_-})}-C\varepsilon \ge \rho_{(1-\alpha_-)c}-C\varepsilon.
\end{align*}
Hence
\begin{equation}\label{denom}
\frac{1}{\P(G(a_n)<C(a_n))}\le \frac{1}{\rho_{(1-\alpha_-)c}}+C\varepsilon.
\end{equation}
Putting equations \eqref{upperbound}, \eqref{eq:espSp}, \eqref{eq:espSX} and \eqref{denom} together
\begin{align*}
&\E\left[\sum_{j=1}^{G(a_n)} W_j \, \middle| \, G(a_n)<C(a_n)\right]\\
&\le \frac{(1-p_{\alpha_+})}{\rho_{(1-\alpha_-)c}}\left(1-(1-\alpha_-)c(1-\rho_{(1-\alpha_-)c})\right)\left(\frac{1}{1-(1-\rho_{(1-\alpha_-)c})(1-\alpha_-)c}\right)+C\varepsilon\\
&\le \frac{1-\rho_{(1-\alpha_-)c}}{\rho_{(1-\alpha_-)c}}+C\varepsilon.\end{align*}
Recalling 
$$\tau_{i+1}-\tau_i=1+2\sum_{j=1}^{G(a_n)} W_j,$$
we get the desired result.
\end{proof}

\subsection{Concentration for the pinning heights}\label{cht}

The sharp estimate of the length of the renewal intervals obtained in Lemma \ref{espcond} converts into concentration for the pinning heights $(h_k)_{1 \leq k \leq K}$ defined by \eqref{defhk}:
\begin{lemma}\label{lemdis}
There exists a constant $C$, depending only on $\eta$, such that for every $k \leq K$, with high probability,
$$\varepsilon \rho_{(1-k\varepsilon)c}-C\varepsilon^2 \le  \liminf_{N \to \infty} \frac{h_{k+1}-h_k}{N}\le \limsup_{N \to \infty} \frac{h_{k+1}-h_k}{N}\le \varepsilon \rho_{(1-k\varepsilon)c}+  C\varepsilon^2.$$
\end{lemma}
\begin{proof}
Fix $k\le K$, we are going to construct a martingale involving the sequence $(\tau_i)_{h_k\le i < h_{k+1}}$. 
Recall that on $\mathbf{G},$ the sequence $(\tau_i)$ is a Markov chain, and that $h_k$ is a stopping time for it. Indeed, as we saw earlier, $\tau_{i+1}-\tau_i$ has an explicit distribution, depending only on $\tau_i+i$.

We modify slightly the sequence $\tau_{h_k+i}$ in the following way. Let $\tilde{\tau}_{h_k+i}$ be equal to $\tau_{h_k+i}$ as long as $\tau_{h_k+i}-\tau_{h_k}+i \le 2\varepsilon N$. Then complete the sequence by adding to the last term $\tau_{h_{k+1}}$ i.i.d. copies of $\tau_{h_{k}+1}-\tau_{h_k}$ at each step. This is just a formal definition, and we are only interested in $h_{k+1}-h_k$, which is precisely, by definition, the hitting time of $2\varepsilon N$ by the sequence $\tau_{h_k+i}-\tau_{h_k}+i$. Obviously changing the sequence after this hitting time won't modify it.

Now we introduce the martingale $M^{(k)}_n$ with respect to $\sigma(\tilde{\tau}_{h_k+i})_{i\ge 0}$ defined on $\mathbf{G}$ by
\begin{equation*}
\begin{cases}
M^{(k)}_0=0, \\
M^{(k)}_{n}=\tilde{\tau}_{h_{k}+n}-\tilde{\tau}_{h_k}- \sum_{i=0}^{n-1} \E[\tilde{\tau}_{h_k+i+1}-\tilde{\tau}_{h_k+i}|\tilde{\tau}_{h_k+i}].
\end{cases}
\end{equation*}
Recall that, still on the event $\mathbf{G}$, the difference $|\tilde{\tau}_{h_k+i+1}-\tilde \tau_{h_k+i}|$ is smaller then $N^{1/10}\log N$, while by construction and Lemma \ref{espcond}, for every $i \geq 0$,
$$\frac{2}{\rho_{(1-\alpha_-)c}}-1-C\varepsilon\le \E[\tilde\tau_{h_k+i+1}-\tilde \tau_{h_k+i}|\tilde \tau_{h_k+i}]\le \frac{2}{\rho_{(1-\alpha_-)c}}-1+C\varepsilon.$$
Therefore, for $N$ large enough, the increments of $M^{(k)}$ are bounded by $N^{1/10}\log N$.

Azuma-Hoeffding inequality gives that
\[
\P \left( M^{(k)}_{n} >N^{3/4} \right)\le 2 \exp \left(-2 \frac{(N^{3/4})^2}{N (N^{1/10} \log N)^2} \right) \leq C \exp(-N^{1/4}),
\]
therefore, by the union bound,
\[
\P \left( \sup_{n\le \varepsilon N} M^{(k)}_{n}>N^{3/4} \right)\le \varepsilon \, C N \exp(-N^{1/4})
\]
and, since $K \leq C/\varepsilon$, using once again the union bound
\[
\P \left( \sup_{k\le K}\sup_{n\le \varepsilon N} M^{(k)}_{n}>N^{3/4} \right) \le CN \exp(-N^{1/4}),
\]
which can be made as small as requested by taking $N$ large.

This implies that,  as $N\to \infty$, with high probability
\[
\left|\tilde{\tau}_{h_{k}+n}-\tilde{\tau}_{h_k}- \sum_{i=0}^{n-1} \E[\tilde{\tau}_{h_k+i+1}-\tilde{\tau}_{h_k+i}|\tilde{\tau}_{h_k+i}]\right|<N^{3/4}, 
\]
whence, for all $n\le \varepsilon N$,
$$n \left(\frac{2}{\rho_{(1-\alpha_-)c}}-C\varepsilon\right)\le\tilde{\tau}_{h_{k}+n}-\tilde{\tau}_{h_k}+n\le n \left(\frac{2}{\rho_{(1-\alpha_-)c}}+C\varepsilon\right).$$
Recalling that $h_{k+1} - h_k$ is the hitting time of $2 \varepsilon N$ by the sequence $(\tilde{\tau}_{h_{k}+n}-\tilde{\tau}_{h_k}+n)_n$, we get the result.
\end{proof}

\subsection{Proof of Theorem~\ref{th:main}}\label{pinning}
As we are now going to manipulate $\varepsilon$, we keep track of the dependency of the $h_k$'s on $\varepsilon$ by writing $h^\varepsilon_{k}$.
As a consequence of Lemma \ref{lemdis}, for every $k\le K=(1-(1+\eta)/c)/\varepsilon$ 
$$\sum_{i=0}^{k-1}\varepsilon\rho_{(1-i\varepsilon)c}-kC\varepsilon^2\le \liminf_{N\to\infty} \frac{h^{\varepsilon}_k}{N} \le \limsup_{N\to\infty}\frac{h^{\varepsilon}_k}{N} \le\sum_{i=0}^{k-1} \varepsilon\rho_{(1-i\varepsilon)c}+kC\varepsilon^2.$$
Taking $k=\lceil u/\varepsilon\rceil$, we identify a Riemann sum, so that by derivability of the integrated function $x \mapsto \rho_x$, uniformly in $u \in [0,1-(1+ \eta)/c]$
\begin{equation}
\label{eq:hRiemann}
\int_0^u \rho_{(1-x)c}dx-C\varepsilon \le \liminf_{N\to\infty} \frac{h_{\lceil u/\varepsilon\rceil}}{N}\le\limsup_{N\to\infty} \frac{h_{\lceil u/\varepsilon\rceil}}{N}\le\int_0^u \rho_{(1-x)c}dx+C\varepsilon.
\end{equation}

The $K$ points of the normalized profile $(n/N,X_n/N)$ of the walk taken at the times $\tau_{h_k^\varepsilon}$ for $k \in \{1,\ldots,K\}$, can be written
\begin{align}
\label{eq:pinned}
\left( \frac{\tau_{h_k^\varepsilon}}{N} , \frac{X_{\tau_{h_k^\varepsilon}}}{N} \right) & = \left( \frac{\tau_{h_k^\varepsilon}}{N} , \frac{h_k^\varepsilon}{N} \right) = \left( 2 \, k \, \varepsilon + \mathcal{O}\left( N^{-4/5} \right)- \frac{h_k^\varepsilon}{N}, \frac{h_k^\varepsilon}{N} \right);
\end{align}
the last equality coming from the fact that, on the event $\mathbf G$, each increment $\tau_{i+1} -  \tau_i$ is bounded from above by $N^{1/10} \, \log N$.

Gathering \eqref{eq:hRiemann} and \eqref{eq:pinned}, we obtain that as $N \to \infty$, these $K$ points of the normalized profile of the walk are uniformly at distance $ C \varepsilon$ of the following parametrized curve:
\begin{equation}
\label{eq:profileup}
\begin{cases}
x(u) = 2u -\int_0^u \rho_{(1-x)c} \, \mathrm{d}x,\\
y(u) = \int_0^u \rho_{(1-x)c} \, \mathrm{d}x.
\end{cases}
\end{equation}
Finally, recalling that
\[
\left(\frac{\tau_{h_{k+1}^\varepsilon}}{N} - \frac{\tau_{h_k^\varepsilon}}{N} \right) + \left( \frac{X_{\tau_{h_{k+1}^\varepsilon}}}{N} - \frac{X_{\tau_{h_k^\varepsilon}}}{N} \right) = 2 \varepsilon
\]
and that the slope of the renormalized profile is smaller than $1$ in absolute value, we are assured that the whole normalized profile stays at distance smaller than $C \varepsilon$ from the curve defined by \eqref{eq:profileup}. Taking $\varepsilon \to 0$ first and then $\eta \to 0$, we have the convergence of the normalized profile of the walk for the parameter $u$ ranging from $0$ to $1-1/c$.

To identify the parametrized curve defined by \eqref{eq:profileup} with the explicit one given in Theorem \ref{th:main}, we just have to parametrize the curve by  $\rho_{(1-u)c}$ instead of $u$. The definition \eqref{relrho} of $\rho_{(1-u)c}$ gives
\[
u = 1 + \frac{\log \left( 1- \rho_{(1-u)c} \right)}{c \, \rho_{(1-u)c}}.
\]
From this relation, we can proceed to a change of variable in the integral appearing in \eqref{eq:profileup} and get the announced formulas.

\bigskip

We now turn to the convergence of the profile of the process after reaching criticality, that is during the way down.

For every $k\leq K$, we introduce
$$\zeta_k=\inf\{n\ge \tau_{h_k+1}:a_n=a_{\tau_{h_k}}\}$$
the time when the walker returns to its position at time $\tau_{h_k}$ after exploring the connected component of $a_{\tau_{h_k+1}}$ in $S_{\tau_{h_k}}$. On $\textbf{G}$, the difference $\zeta_k-\tau_{h_k}$ is twice the size of this connected component. Recall that, on $\textbf{G}$, the subgraph $S_{\tau_{h_k}}$ is an Erd\H{o}s-Rényi graph with number of vertices in $(1-k\varepsilon)N+O(N^{1/5})$ and connection probability $c/N$. As a consequence, for every $k$, 
$$\lim_{N\to\infty} \frac{\zeta_k-\tau_{h_k}}{N}= (1-k\varepsilon) \rho_{(1-k\varepsilon)c}.$$
Besides $X_{\zeta_k}=X_{\tau_{h_k}}=h_k$. This implies that the K points of the profile taken at times $\zeta_k$ for $k \in \{1,\ldots,K\}$ can be written
\begin{align*}
\left( \frac{\zeta_k}{N} , \frac{X_{\zeta_k}}{N} \right) &= \left(\frac{\tau_{h_k}}{N}+(1-k\varepsilon) \rho_{(1-k\varepsilon)c}+o(1), \frac{h_k^\varepsilon}{N} \right);
\end{align*}
where the term $o(1)$ goes to zero as $N\to\infty$.

Using the same slope arguments as before, we get the announced parametrization. 

\bigskip

\subsection*{Acknowledgments}
N.E. is partially supported by ANR PPPP (ANR-16-CE40-0016). N.E. and G.F. are partially supported by ANR MALIN. L.M. is partially supported by ANR GRAAL (ANR-14-CE25-0014).

All three authors acknowledge the support of Labex MME-DII
(ANR11-LBX-0023-01).

\bigskip

\noindent \textsc{Nathana\"el Enriquez:} \verb|nathanael.enriquez@u-paris10.fr|, \\
\textsc{Gabriel Faraud:} \verb|gabriel.faraud@u-paris10.fr|,\\
\textsc{Laurent M\'enard:} \verb|laurent.menard@normalesup.org|.

\bigskip

\noindent \textsc{Modal'X, UPL, Univ. Paris Nanterre, F92000 Nanterre France}


\begin{bibdiv}
\begin{biblist}[\normalsize]

\bib{AKS}{article}{
    AUTHOR = {Ajtai, Mikl\'os},
    AUTHOR = {Koml\'os, J\'anos},
    AUTHOR = {Szemer\'edi, Endre},
     TITLE = {The longest path in a random graph},
   JOURNAL = {Combinatorica},
    VOLUME = {1},
      YEAR = {1981},
    NUMBER = {1},
     PAGES = {1--12},
review={\MR{602411}},
}

\bib{B}{book}{
   author={Bollob\'as, B\'ela},
   title={Random graphs},
   series={Cambridge Studies in Advanced Mathematics},
   volume={73},
   edition={2},
   publisher={Cambridge University Press, Cambridge},
   date={2001},
   pages={xviii+498},
   review={\MR{1864966}},
}
\bib{D}{book}{
   author={Durrett, Rick},
   title={Random graph dynamics},
   series={Cambridge Series in Statistical and Probabilistic Mathematics},
   volume={20},
   publisher={Cambridge University Press, Cambridge},
   date={2007},
   pages={x+212},
   review={\MR{2271734}},
}

\bib{E}{article}{
   author={Erd\H os, Paul},
   title={Problems and results on finite and infinite graphs},
   conference={
      title={Recent advances in graph theory},
      address={Proc. Second Czechoslovak Sympos., Prague},
      date={1974},
   },
   book={
      publisher={Academia, Prague},
   },
   date={1975},
   pages={183--192. (loose errata)},
   review={\MR{0389669}},
}

\bib{ER}{article}{
   author={Erd\H os, Paul},
   author={R\'enyi, Alfred},
   title={On the evolution of random graphs},
   language={English, with Russian summary},
   journal={Magyar Tud. Akad. Mat. Kutat\'o Int. K\"ozl.},
   volume={5},
   date={1960},
   pages={17--61},
   review={\MR{0125031}},
}

\bib{F}{book}{
   author={Frieze, Alan},
   author={Karo\'nski, Michal},
   title={Introduction to Random Graphs},
    publisher={Cambridge University Press, Cambridge},
   date={2016},
   isbn={9781107118508},
}

\bib{KS}{article}{
   author={Krivelevich, Michael},
   author={Sudakov, Benny},
   title={The phase transition in random graphs: a simple proof},
   journal={Random Structures Algorithms},
   volume={43},
   date={2013},
   number={2},
   pages={131--138},
   review={\MR{3085765}},
}

\bib{V}{article}{
   author={de la Vega, Fernandez},
   title={Sur la plus grande longueur des chemins élémentaires de graphes aléatoires},
   journal={Preprint of Laboratoire d’informatique pour les Sciences de l’Homme, C.N.R.S},
   date={1979},
}


\bib{RW}{article}{
   author={Riordan, Oliver},
   author={Wormald, Nicholas},
   title={The diameter of sparse random graphs},
   journal={Combin. Probab. Comput.},
   volume={19},
   date={2010},
   number={5-6},
   pages={835--926},
   review={\MR{2726083}},}		
\end{biblist}
\end{bibdiv}
\end{document}